\documentclass[a4paper,11pt]{article}

\usepackage{amsmath}
\usepackage{amsthm}
\usepackage{verbatim}

\usepackage[top=3.5cm, bottom=3.5cm, left=2.5cm, right=2.5cm]{geometry}

\usepackage[utf8]{inputenc}
\usepackage[T1]{fontenc}
\usepackage[french,english]{babel} 
\usepackage{enumerate}

\usepackage{xcolor}
\usepackage[]{pdfpages}
\usepackage{command-estim}
\usepackage{titling}
\usepackage[colorlinks=false,linkcolor=blue]{hyperref}
\usepackage{csquotes}
\usepackage[nottoc, notlof, notlot, numbib]{tocbibind}
\usepackage[style=alphabetic,giveninits,doi=false,isbn=false,url=false,eprint=false,clearlang=true,sorting=nyt,maxbibnames=99]{biblatex}
\DeclareSourcemap{
      \maps[datatype=bibtex]{
      \map[overwrite=false]{
       \step[fieldsource=note]
       \step[fieldset=addendum, origfieldval, final]
       \step[fieldset=note, null]
       }
   }
   }

\usepackage{mathtools}
\usepackage{thmtools}
\usepackage{bbm}
\usepackage{amssymb}
\usepackage{mathrsfs}
\usepackage{float}
\usepackage{stmaryrd}

\usepackage{setspace}
\usepackage{array}
\usepackage{tabularx}
\usepackage{ltablex} 

\usepackage{todonotes}
\usepackage{mathtools}
\mathtoolsset{showonlyrefs}

\newcommand{\normalcone}[2]{{\mathcal N}_{#1}(#2)}

\usepackage{authblk}

\title{\bf Constrained non-linear estimation and links with stochastic filtering}

\begin{document}

\author[1]{Louis-Pierre \textsc{Chaintron}}
\author[2]{Laurent \textsc{Mertz}}
\author[3]{Philippe \textsc{Moireau}}
\author[4]{Hasnaa \textsc{Zidani}}
\affil[1]{\small
DMA, École normale supérieure, Université PSL, CNRS, 75005 Paris, France}
\affil[2]{\small
Department of Mathematics, City University of Hong Kong, Kowloon, Hong Kong, China \& City University of Hong Kong Shenzhen Research Institute, Shenzhen, China
}
\affil[3]{\small
Ecole Polytechnique, 91128 Palaiseau, France
}
\affil[4]{\small
INSA Rouen Normandie, 76800, Saint-Étienne-du-Rouvray, France
}

\date{}

\maketitle

\abstract{This article studies the problem of estimating the state variable of non-smooth sub-differential dynamics constrained in a bounded convex domain given some real-time observation. On the one hand, we show that the value function of the estimation problem is a viscosity solution of a Hamilton Jacobi Bellman equation whose sub and super solutions have different Neumann type boundary conditions. This intricacy arises from the non-reversibility in time of the non-smooth dynamics,  
and hinders the derivation of a comparison principle and the uniqueness of the solution in general. Nonetheless, we identify conditions on the drift (including zero drift) coefficient in the non-smooth dynamics that make such a derivation possible. On the other hand, we show in a general situation that the value function appears in the small noise limit of the corresponding stochastic filtering problem by establishing a large deviation result. We also give quantitative approximation results when replacing the non-smooth dynamics with a smooth penalised one.}

\tableofcontents

\section{Introduction} \label{sec:intro}

Sequential estimation aims to combine a dynamical system with some measurements as they become available, to reduce potential uncertainties in the dynamics and thus produce a model prediction that is more consistent with available data. 
Such a goal can be pursued for a wide range of dynamical systems: finite dimensional (ODE) or infinite dimensional (PDE), linear or non-linear dynamics, unconstrained (smooth dynamics) or constrained (non-smooth dynamics combined with variational inequality), deterministic (observer theory) or stochastic (filtering theory). 
In the present paper, we focus on constrained dynamics in finite dimension, and we link the framework of stochastic filtering with the deterministic framework aiming at defining observer dynamics. 
Stochastic filtering for unconstrained ODEs has been known since the 1960s with the seminal work of \cite{kalmanbucy1961} for linear dynamics, and then generalized for nonlinear dynamics, see \cite{kushner1967dynamical,duncan1967probability,zakai1969optimal,jazwinski1970stochastic}...
These results were then extended to some constrained dynamics, in particular for the Skorokhod problem with the series of works by \cite{pardoux1978filtrage,pardoux1978stochastic,hucke1990nonlinear}. 
As for the deterministic view, the observer theory based on Minimum Energy Estimation has been known since the pioneering work of \cite{mortensen1968maximum}, see also the presentation proposed by \cite{fleming1997deterministic}.
While the unconstrained case has been well understood since \cite{barasjames}, with an asymptotic connection to stochastic filtering introduced in \cite{HijabAoP} and further justified in \cite{fleming1997deterministic}, the case of constrained dynamics was not studied until a recent attempt \cite{chaintron2023mortensen} for a simple one-dimensional dynamics. 
The constraint is there introduced using the formalism of non-smooth sub-differential dynamics \cite{moreau1971rafle,tyrrell1970convex}.
A main difficulty of this setting is the loss of time reversibility: in contrast with smooth dynamics, the non-smooth dynamics is well-posed in forward time only, making the connection harder between stochastic filtering and deterministic estimation.
However, \cite{chaintron2023mortensen} was able to make a few strides to reconcile both points of view. 
 
In this paper, we generalize the works of \cite{willems2004deterministic,barasjames,chaintron2023mortensen} to non-smooth dynamics associated with trajectories that must remain in a bounded domain. 
In particular, we fully connect the deterministic representation to stochastic filtering by extending the results of \cite{barasjames} in a suitable way: we show that the deterministic view corresponds to the small noise limit of the stochastic framework by proving a large deviation result. 
Following \cite{barasjames,fleming1997deterministic},
we rely on a viscosity solution setting to deal with the underlying Hamilton-Jacobi-Bellman (HJB) equations.
However, due to the non-reversibility of non-smooth dynamics, our approach differs from them.
In particular, the sub-solution and the super-solution satisfy different Neumann-type boundary conditions in the viscosity sense, and the comparison principle is unknown to our knowledge.
To circumvent this difficulty, the small noise limit is established using a dual formulation.
We also show how the penalized estimator developed in \cite{chaintron2023mortensen} to deal with the constraint converges to the fully constrained estimator in a quantitative way, and we provide numerical illustrations.

\subsection{Problem statement}
\label{subsec:problem}

Let $\stateSpace \subset \RR^\NstateSpace$ be a bounded open domain that is convex with $C^{2}$ boundary. 
We consider a class of non-smooth dynamical systems of the form
\begin{equation} \label{eq:dynsys}
\dot\state_{\modelNoise} (s) + \partial\model (\state_{\modelNoise} (s)) \ni \modelRHS(s,\state_{\modelNoise} (s)) + \sigma(s,\state_{\modelNoise} (s)) \modelNoise(s),\quad s>0,
\end{equation}
with an initial condition $\state_{\modelNoise} (0) \in \overline{\stateSpace}$.

Here, the state variable at time $s$ is denoted $\state_{\modelNoise} (s) \in \mathbb{R}^n$, while the functions $b : \RR_+ \times \overline\stateSpace \to \RR^{\NstateSpace}$ and $\sigma: \RR_+ \times \overline\stateSpace \to \RR^{\NstateSpace \times \NmodelNoiseSpace}$ are Lipschitz-continuous functions. Thus, these functions can be extended in a Lipschitz-continuous manner to the entire space, and we use the same notation for the extensions. The term $\modelNoise(s) \in \RR^{\NmodelNoiseSpace}$ represents the state disturbance, 
 $\model : \RR^{\NstateSpace} \to \{ 0, +\infty \}$ is the  characteristic function of the domain $\stateSpace$, and 
$\partial\model (\state_{\modelNoise} (s))$ is the  subdifferential of the convex function $\model$ at $\state_{\modelNoise} (s)$. Recall that this subdifferential corresponds to the normal cone $\normalcone{\overline{\stateSpace}}{\state_{\modelNoise} (s)}$
 to $\overline{\stateSpace}$ at $\state_{\modelNoise} (s)$. 
This means that when the state $\state_{\modelNoise} (s)$ reaches the boundary of $\stateSpace$, the dynamics are reflected, ensuring that $\state_{\modelNoise} (s)$ remains within $\overline{\stateSpace}$. In particular, the dynamics are driven by the vector field $\modelRHS(s,\state_{\modelNoise} (s)) + \sigma(s,\state_{\modelNoise} (s)) \modelNoise(s)$, but whenever $\state_{\modelNoise} (s)$ approaches the boundary of $\stateSpace$, the normal cone prevents it from leaving the domain, reflecting the trajectory back into the domain in the sense of Skorokhod.

In the sequel, we assume that the disturbance belongs to the Lebesgue space $L^2(0,+\infty; \R^r)$ that consists of all measurable functions 
$w:(0,+\infty) \longrightarrow \R^r$ that are square integrable, i.e. the norm
$$
\|w\|_{L^2(0,+\infty; \R^r)}:=
	\left(\int_{0}^{+\infty}\|w(s)\|^2\,ds\right)^{\frac12}, 
$$	
is finite, and where functions which agree almost everywhere are identified.

The controlled system is well posed. Indeed, following similar arguments as in  \cite[Theorem~1]{EdmondThibault2005}, one can prove that for any square-integrable disturbances $\modelNoise$, for  any initial data $x_\omega(0)\in\overline{G}$, the system \eqref{eq:dynsys}  admits a unique absolutely continuous 
solution $x_\omega$ 
such that  the differential inclusion \eqref{eq:dynsys} holds for almost every $s \geq 0.$

We consider a measurement procedure $\obsMap : \RR_+ \times \stateSpace \rightarrow \RR^m $, so that observations associated with a trajectory of the dynamics \eqref{eq:dynsys} are given by
\begin{equation}
\label{eq:obs}
	\forall t \geq 0, \quad \observ(t) = \obsMap(t,\state_\modelNoise(t)) + \obsNoise(t),
\end{equation}
where $\obsNoise(t) \in \RR^m$ is the observation \textit{disturbance}, which is assumed to be square-integrable. 
Our purpose is to build an estimator at time $t$ for the state of a dynamics described by \eqref{eq:dynsys} given the measurement $( \observ(s) )_{0 \leq s \leq t}$ produced by \eqref{eq:obs}. 
We want this estimator to be \emph{causal} in the sense of \cite{Krener:1998ve}, meaning that the computed state only depends on the past measurements.

\begin{remark}
	In most deterministic observation problems, the observation is usually denoted by $y$. 
    Here, we denote by $y$ a primitive of the observation
	for consistency with the stochastic filtering setting. 
    As in practice, only $\observ$ will appear in the estimator equations, and defining the observation with a time-derivative is a mere notation convention.
\end{remark}

An iconic case of non-smooth dynamics in an unbounded domain is $\stateSpace = \RR_+$, studied in the context of estimation in \cite{chaintron2023mortensen}.

Before describing our results, we briefly review some existing methods in linear and non-linear estimation.
\subsection{Unconstrained linear dynamics: Kalman filter} \label{subsec:IntroKalman}

In the linear case, the system is characterized by the following dynamics:
\[
b(t,x) = A x, \quad \sigma(t,x) = \Sigma, \quad h(t,x) = H x, \quad \psi(x) = \frac{1}{2} [x - \hat{x}_0]^\top P_0^{-1} [x - \hat{x}_0],
\]
where $A \in \R^{n \times n}$, $\Sigma \in \R^{n \times r}$, $H \in \R^{m \times n}$, $P_0 \in S^{++}_n (\R)$ is a positive-definite matrix, and $\hat{x}_0 \in \R^n$. The unconstrained linear dynamics, with a linear observation operator, serves as a classical example for defining a sequential estimator in both deterministic and stochastic systems.
This was first introduced by \cite{kalmanbucy1961}, and a full treatment of stochastic filtering in this case can be found in \cite{jazwinski1970stochastic,davis1977linear}. For a deterministic perspective, we refer to \cite{willems2004deterministic}. Essentially, both approaches lead to the definition of the estimator given by
\begin{equation}\label{eq:kalman}
\begin{cases}
    \dot{\hat{x}}(t) = A \hat{x}(t) + P(t) H^\top [ \dot{y}(t) - H \hat{x}(t) ], \quad t > 0, \\
    \hat{x}(0) = \hat{x}_0,
\end{cases}
\end{equation}
where the symmetric positive matrix $P \in S_\NstateSpace^+(\R)$ is a solution of the following Riccati equation
\begin{equation}\label{eq:Riccati}
\begin{cases}
\dot{P}(t) = A P(t) + P(t) A^\top + \Sigma \Sigma^\top - P (t)^\top H H^\top P(t), \quad t > 0, \\
P(0) = P_0.
\end{cases}
\end{equation}
Equation \eqref{eq:kalman} provides a recursive estimator that can be computed in real time, the matrix $P(t)$ being pre-computed beforehand.

\subsection{Non-linear dynamics: Mortensen observer} \label{subsec:IntroMortensen}

If we now consider a general non-linear unconstrained dynamics of the form
\begin{equation}
\label{eq:dynsysUncons}
\dot\state_{\modelNoise} (s) = \modelRHS(s,\state_{\modelNoise} (s)) + \sigma(s,\state_{ \modelNoise} (s)) \modelNoise(s),\quad s > 0,
\end{equation}
the Mortensen estimator \cite{mortensen1968maximum} generalises the Kalman estimator in the deterministic setting. 
Given an observation $ \obsVar \in \Ltwo[(0,t);\R^{m} ]$, we introduce a value function called \emph{cost-to-come} as the function defined by 
\begin{equation}
\label{eq:costcomeIntro}
\costcomeIntro (t,x) \triangleq \inf_{\substack{\modelNoise \in L^2((0,t);\RR^r) \\ \state_{\modelNoise}(t) = x}} \initError(x_\modelNoise(0)) + \int_0^t \timeError( s,\state_{\modelNoise}(s),\modelNoise(s) )\, \diff s ,
\end{equation}
where $\psi : \RR^n \rightarrow \R_+$ is Lipschitz and such that $\psi (x) \rightarrow + \infty$ as $\vert x \vert \rightarrow +\infty$, 
\begin{equation}
\label{eq:ell_formula}
\timeError( s,\state,\modelNoise ) \triangleq \frac{1}{2} \lvert \modelNoise \rvert^2 + \frac{1}{2 } \lvert \obsVar(s) - \obsMap(s,\state) \rvert^2.
\end{equation}

The Mortensen observer $\hat{\state}$ is defined as a minimiser of $x \mapsto \costcomeIntro (t,x)$.
If uniqueness holds for this minimiser, we simply define
\begin{equation}\label{eq:mortensen}
\forall t\geq 0, \quad  \hat{\state}(t) \triangleq \argmin_{\state \in \R^\NstateSpace} \costcomeIntro (t,x).
\end{equation}
For a well-posedness result in a non-linear setting, we refer to \cite{breiten2023local}.
If moreover $\costcomeIntro (t,x)$ is $C^2$ at $(t,\hat{\state}(t))$ with invertible Hessian, the optimality conditions for $\hat{\state}(t)$ yield
\begin{equation} \label{eq:mortensenrec}
\dot{\hat{x}}(t) = b(t,\hat{x}(t)) + [ \nabla^2 \costcomeIntro(t,\hat{\state}(t)) ]^{-1} \nabla h(t,\hat{x}(t)) [ \dot{y}(t) - h(t, \hat{x}(t)) ], \quad t > 0. 
\end{equation} 
This recursive feed-back structure extends the one of the Kalman estimator.
Indeed, \eqref{eq:mortensenrec} precisely corresponds to \eqref{eq:kalman} in the linear setting of Section \ref{subsec:IntroKalman}.
In this linear setting, the cost-to-come reads
\[ \costcomeIntro(t,x) = \frac{1}2  [x - \hat{x}(t)] \cdot P^{-1}(t) [x - \hat{x}(t)] + \int_0^t \frac{1}2 |\dot{y}(s) - H x(s)|^2 \diff s, \]
enabling us to recover the Kalman estimator.

\subsection{Outline of the article} \label{subsec:IntroOutline}  

In this paper, we aim at adapting Mortensen's approach to the framework of non-smooth dynamics as described in \eqref{eq:dynsys}. In this setting, the normal cone to $\overline{G}$ is involved, and the system dynamics belongs to the class of problems known as sweeping dynamics. This class has attracted considerable interest in recent years within the optimal control community, as it models systems with state constraints that evolve along trajectories governed by set-valued dynamics, see for example 
\cite{Colomboetal} and the references therein. Notable contributions to the study of sweeping dynamics include works on the Pontryagin Maximum Principle (PMP), such as those in \cite{Mordukovich, de2019optimal, Zeidan}, and on the characterization of the value function through HJB approach, as investigated in \cite{Herm2024} and the references therein.

Despite this progress, a major challenge in this domain arises from the irreversible nature of sweeping process trajectories. Unlike classical dynamical systems, these trajectories do not allow for direct and backward time evolutions to be treated symmetrically. This temporal asymmetry complicates both the theoretical analysis and the practical computation of solutions, particularly in the context of backward reachability or filtering problems.

To the best of our knowledge, the analysis of sweeping processes within the context of filtering has yet to be fully explored. By extending Mortensen's framework to such non-smooth dynamics, we aim to address this gap and provide some new insights into the interplay between set-valued dynamics and optimal filtering techniques.

The paper is organized as follows. The cost-to-come function for the constrained setting is defined in Section \ref{subsec:ResPen} and characterised as a viscosity solution with intricate boundary conditions.
An approximation procedure is further introduced to bring the problem back to the setting of Section \ref{subsec:IntroMortensen}, the sub-differential being replaced by a penalisation term that pulls the dynamics back in $\overline{G}$ when it escapes. 
A quantitative rate of convergence is obtained, the proof being done in Section \ref{subsec:CVpen}.
We then describe the link with non-smooth stochastic filtering.
Existing links between stochastic filtering and deterministic estimation are recalled in Sections \ref{subsubsec:KalmanSto}-\ref{subsubsec:BarasJames}.
Our results for the non-smooth setting are stated in Section \ref{subsubsec:SmallReflected} and proved in Section \ref{sec:Limitfiltering}.
A numerical illustration is given in the supplementary materials.

\section{Description of the results} \label{sec:results}

Throughout the paper, we assume $\sigma(x)^\top \sigma(x)$ to be \emph{uniformly invertible}, meaning that there exists a constant $\gamma_0 > 0$ such that for every $x \in \overline{G}$, the following condition holds:
\begin{equation}\label{eq.inversion_sigma}
\sigma(x)^\top \sigma(x) \geq \gamma_0 \, I_d,
\end{equation}
where $I_d$ denotes the identity matrix of appropriate dimension. This condition ensures that the matrix $\sigma(x)^\top \sigma(x)$ is uniformly positive definite, implying that $\sigma(x)$ has full column rank for all $x \in \overline{G}$.
\medskip

Following \cite{James:1988vc,rockafellar2000convexity}, we introduce the following value function. 

\begin{definition}[Cost-to-come]
The \emph{cost-to-come} to the point $\state \in \overline{G}$ at time $t \geq 0$, given an observation $\obsVar \in \Ltwo[(0,t);\RR^m]$, is the function defined by 
\begin{equation}
\label{eq:costcome}
\costcome(t,x) \triangleq \inf_{(x_\modelNoise(0), \modelNoise) \in \mathcal{A}^\stateSpace_{t,x}} \initError(x_\modelNoise(0)) + \int_0^t \timeError(s,\state_{\modelNoise}(s),\modelNoise(s))\, \diff s,
\end{equation}
where $\timeError(s,\state,\modelNoise )$ is defined as in \eqref{eq:ell_formula}
and the admissible set is defined by
\[ \mathcal{A}^\stateSpace_{t, x} \triangleq \{ (x_\modelNoise(0), \modelNoise) \in \overline{G} \times \Ltwo[(0,t);\R^r] : \state_{\modelNoise} \text{ follows } \eqref{eq:dynsys} \text{ with } \state_{\modelNoise}(t) = \state \}. \]
\end{definition}

The uniform invertibility of $\sigma^\top \sigma$, as stated in assumption \eqref{eq.inversion_sigma},  is a sufficient condition for $\mathcal{A}_{t,x}$ to be non-empty, because the convexity of $\overline{G}$ allows connecting any pair of points in $\overline{G}$ by a straight line without escaping $\overline{G}$. Note also that for any $(x_\modelNoise(0), \modelNoise)\in \mathcal{A}^\stateSpace_{t, x}$,  there exists a unique trajectory $x_\modelNoise$ that is absolutely continuous on $[0,T]$ for every 
$T>0$ (and by definition of $\mathcal{A}^\stateSpace_{t, x}$, $\state_{\modelNoise}(t) = \state$).

Before studying $\costcome$ itself, we describe an approximating procedure that brings the problem back to smooth unconstrained dynamics. The continuity of $\costcome$ will be a consequence of this approximation as shown in Theorem \ref{thm:CVCostpen} below.

\begin{remark}[Inward pointing drift] \label{rem:inward}
Let us assume that $b (x) \cdot n(x) \leq 0$ for every $x \in \partial G$, $n (x)$ being the outward normal vector at $x$.
For simplicity, we further assume that $n = r$ 
and $\sigma \equiv \mathrm{Id}$.
In this setting, the control problem \eqref{eq:costcome} of the differential inclusion \eqref{eq:dynsys} reduces to a standard control problem for a differential equation under state constraints. 
Indeed, introducing the solution 
$z_\omega$ of 
\begin{equation} \label{eq:eqNormal}
\dot{z}_\omega (s) = b (s, {z}_\omega (s) ) + \omega (s),  
\end{equation}
we notice that
\begin{equation*}
\costcome(t,x) = \inf_{\substack{\modelNoise, \, z_\omega (t) = x \\
\forall s \in [0,t], \, z_\omega (s) \in \overline{G}}}  \initError(z_\modelNoise(0)) + \int_0^t \timeError(s,z_{\modelNoise}(s),\modelNoise(s))\, \diff s.
\end{equation*}
The $\leq$ inequality stems from the fact that any trajectory of \eqref{eq:eqNormal} that stays in $\overline{G}$ is a trajectory of \eqref{eq:dynsys}. 
The $\geq$ inequality results from the fact that any trajectory of \eqref{eq:dynsys} can be realised by a trajectory of $\eqref{eq:eqNormal}$ for a $\omega$ that has a lower $L^2$-norm, using that $b \cdot n \leq 0$.
\end{remark}

\subsection{Penalisation approach} \label{subsec:ResPen}

In this section, we introduce an estimation problem for an approximation of \eqref{eq:dynsys} defined in the whole space $\RR^\NstateSpace$.
With a slight abuse of notations, we assume that our coefficients $b$, $\sigma$, $\psi$ and $h$ are defined and Lipschitz-continuous on $\RR^\NstateSpace$.
Let $\proj$ denote the orthogonal projection on the closed convex set $\overline\stateSpace$.
For $\penal > 0$, we introduce the penalisation $\regpen : \stateSpace \rightarrow \R^\NstateSpace$ of $\partial\model$ defined by $\regpen (x) \triangleq \penal [ x - \proj (x) ]$.
Since $G$ is convex, we can follow the Moreau-Yosida regularisation \cite{moreau1971rafle}, which was extended to more general domains in \cite{thibault2008regularization,jourani2017moreau}.
We also refer to \cite{de2019optimal} for an alternative smooth exponential penalization method. This approach has been successfully employed to derive optimality conditions for control problems involving sweeping processes.
In our context, we  introduce the penalised dynamics as follows

\begin{equation}
\label{eq:dynsyspen}
\dot\state^\kappa_{\modelNoise} (t) + \regpen( \state^\kappa_{\modelNoise} (t) ) = \modelRHS(t,\state^\kappa_{\modelNoise} (t)) + \sigma(t,\state^\kappa_{\modelNoise} (t)) \modelNoise(t),\quad t > 0.
\end{equation}
The penalised cost-to-come is defined as
\begin{equation}
\label{eq:costcomepen}
\costcomepen (t,x) \triangleq \inf_{\substack{
\modelNoise \in L^2 (0,t), \\
x^\kappa_\omega (t) = x
}} \initError(x^\kappa_\modelNoise(0)) + \int_0^t \timeError( s,\state^\kappa_{\modelNoise}(s),\modelNoise(s) )\, \diff s.
\end{equation}
By standard arguments in control theory \cite[Section II.10]{fleming2006controlled}, $\costcomepen$ is a continuous function.
Our first main results are the following.

\begin{theorem}[Quantitative convergence of $\costcome^\kappa$] \label{thm:CVCostpen}
For every $t > 0$,
\[ \sup_{(s,x) \in [0,t] \times \overline\stateSpace} \lvert \costcomepen (s,x) - \costcome (s,x) \rvert \leq C \kappa^{-1/4}, \]
for some constant $C > 0$ that only depends on $t$.    
\end{theorem}

\begin{corollary}[Convergence of observers] \label{cor:GammaCV}
For every $t > 0$, any limit point of a family $( \hat{x}^\penal (t))_{\penal > 0}$ of minimisers for $x \mapsto \costcomepen(t,x)$ is a minimiser of $x \mapsto \costcome(t,x)$.
\end{corollary}

Theorem \ref{thm:CVCostpen} and its corollary are proved in Section \ref{subsec:CVpen}.
These consistency results show that the estimation problem for \eqref{eq:dynsyspen} is indeed a good approximation of the estimation problem for \eqref{eq:dynsys}, and provide the convergence of Mortensen observers.

\subsection{Hamilton-Jacobi-Bellman equation} \label{subsec:HJB}
Since our estimation procedure is based on the functions $\costcomepen$ and $\costcome$, we need a way to characterise them.
Using the standard dynamic programming approach, the cost-to-come is the unique viscosity solution of a HJB equation.
The notion of viscosity solution is standard to deal with the lack of regularity of solutions to HJB equations, see \cite{crandall1983viscosity,barles1991fully}.
 For the sake of simplicity, we assume in this subsection that 
 \begin{equation}\label{New.assumption}
 \NstateSpace = \NmodelNoiseSpace\qquad \mbox{and} 
 \qquad \sigma \equiv \mathrm{Id}.
 \end{equation}
The following result corresponds to \cite[Theorem 3.1]{barasjames}. 
It is proved in \cite[Section 2]{fleming1997deterministic}, under the additional assumption that $h$ is bounded.
However, the result still holds under our running assumptions, an appropriate comparison principle for unbounded viscosity solutions being provided by \cite{da2006uniqueness}.

\begin{theorem}[Viscosity solution]
$\costcomepen$ is the unique viscosity solution in $(0,T] \times \R^n$ of the HJB equation
\begin{equation*} 
\partial_t \costcomepen (t,x) + (b(t,x)-f_{\kappa}(x)) \cdot \nabla \costcomepen (t,x) + \frac{1}{2} \vert \nabla \costcomepen (t,x) \vert^2 - \frac{1}{2} \vert \dot{y}(t) - h(t,x) \vert^2 = 0, 
\end{equation*}
with the initial condition $\costcomepen(0,x) = \psi(x)$.
\end{theorem}

Formally taking the $\kappa \rightarrow +\infty$ limit, we prove that this equation still holds for $\costcome$ within the open domain $\stateSpace$,
\begin{equation} \label{eq:HJBBulk}
\partial_t \costcome (t,x) + b(t,x) \cdot \nabla \costcome (t,x) + \frac{1}{2} \vert \nabla \costcome (t,x) \vert^2 - \frac{1}{2} \vert \dot{y}(t) - h(t,x) \vert^2 = 0.
\end{equation}
However, intricate boundary conditions now appear. 

\begin{theorem}[Viscosity solution] 
\label{thm:Visco}
The value function $\costcome$ satisfies the HJB equation \eqref{eq:HJBBulk} in the sense described in \cite[Section 2]{barles1991fully},
\begin{itemize}
    \item[(i)] $\costcome$ is a viscosity sub-solution of the HJB equation \eqref{eq:HJBBulk} with the boundary condition
    \begin{equation} \label{eq:BoundSub} 
    b(t,x) \cdot n(x) + \frac{\partial \costcome}{\partial n}(t,x) = 0, \quad x \in \partial\stateSpace. 
    \end{equation}
    \item[(ii)] $\costcome$ is a viscosity super-solution of the HJB equation \eqref{eq:HJBBulk} with the different boundary condition
    \begin{equation} \label{eq:BoundSuper} 
    b(t,x) \cdot n(x) + \frac{1}{2} \frac{\partial \costcome}{\partial n}(t,x) = 0, \quad x \in \partial\stateSpace. 
    \end{equation}
    \item[(iii)] If $b \cdot n \leq 0$ on $\partial \stateSpace$, then $\costcome$ is the unique viscosity solution of the HJB equation \eqref{eq:HJBBulk}-\eqref{eq:BoundSuper}, and the comparison principle holds for \eqref{eq:HJBBulk}-\eqref{eq:BoundSuper}.
\end{itemize}
\end{theorem}

In the boundary conditions \eqref{eq:BoundSub}-\eqref{eq:BoundSuper}, $n(x)$ denotes the outward normal vector at the point $x \in \partial\stateSpace$.
Theorem \ref{thm:Visco} is proved in Section \ref{subsec:visco}.
The used notion of solution is the standard notion of viscosity solution with Neumann boundary condition \cite{lions1985neumann,barles1991fully}, in which the boundary condition is relaxed by allowing the equation to hold at the boundary. 
In general, the boundary conditions \eqref{eq:BoundSub} and \eqref{eq:BoundSuper} do not coincide.
Up to our knowledge, proving a comparison principle for this setting is an open question, which will be the subject of future works. 

However, if $b \cdot n \leq 0$ on $\partial \stateSpace$, the definition of sub-solution shows that any viscosity sub-solution of \eqref{eq:HJBBulk} with the boundary condition \eqref{eq:BoundSub} is also a viscosity sub-solution of \eqref{eq:HJBBulk} with the boundary condition \eqref{eq:BoundSuper}. 
The comparison principle is then proved in \cite[Theorem 3]{barles1991fully} and uniqueness holds for the viscosity solution of the HJB equation \eqref{eq:HJBBulk}-\eqref{eq:BoundSuper}.
Furthermore, Remark \ref{rem:allsuper} shows in this case that $\costcome$ is a viscosity super-solution of \eqref{eq:HJBBulk} on the \emph{closed} domain $[0,T] \times \overline{G}$, with no further boundary condition. 
This recovers the notions of \emph{constrained viscosity solution} introduced in \cite{soner1986optimal,ishii1996new}.
This was expected, since the assumption $b \cdot n \leq 0$ reduces the control problem \eqref{eq:costcome} to the standard control problem of \eqref{eq:eqNormal} under state constraints using Remark \ref{rem:inward}, and \eqref{eq:eqNormal} satisfies the \emph{inward pointing condition} from \cite{soner1986optimal}. 
The related comparison principle is proved in \cite[Theorem III.2]{capuzzo1990hamilton}.
\medskip

Note also that, in the case of optimal control for sweeping processes, a HJB equation was derived in \cite{Herm2024} characterizing a value function for forward-in-time optimal control problems. It is worth noting that this work considers a general process with an admissible set $\overline{G(t)}$ that evolves over time. The HJB equation is formulated with different boundary conditions for sub- and supersolutions, which account for the time-dependent evolution of the set $G$. 
In contrast, in our setting, the set $G$ is fixed, but the state process evolves backward in time.

\subsection{Links with stochastic filtering} \label{subsec:ResExistSto}

In general, viscosity solutions of HJB equations lack regularity; this may hinder an efficient computation of their minimisers. 
To circumvent this, a standard approach is to smooth them by adding a (small) laplacian term. 
Interestingly, \cite{HijabAoP,barasjames,fleming1997deterministic} proved that this smoothing procedure is connected to the renowned stochastic filtering problem
\cite{kallianpur1968estimation,kushner1967dynamical,jazwinski1970stochastic}...
In our constrained setting, this amounts to replacing the model dynamics \eqref{eq:dynsys} by the stochastic reflected dynamics
\begin{equation}
\label{eq:reflectedSDE}
\dd X^\varepsilon_t + \partial\model (X^\varepsilon_t) (\dd t) \ni b(t,X^\varepsilon_t) \dd t + \sqrt{\varepsilon} \sigma (t,X^\varepsilon_t) \dd B_t,
\end{equation}
on a filtered probability space $(\Omega,\mathcal{F},\mathbb{P},(\mathcal{F}_t)_{t \geq 0})$, the law of $X^\varepsilon_0$ being $q^\varepsilon_0 (x) \dd x$. 
The integral of the deterministic model noise $\modelNoise$ in \eqref{eq:dynsys} has been replaced by an adapted Brownian motion $(B_t)_{0\leq t \leq T}$ in $\RR^{\NmodelNoiseSpace}$.
Under our current Lipschitz assumptions, well-posedness for \eqref{eq:reflectedSDE} is proved in \cite[Section 4.2.2]{pardoux2014sdes}.
We complement \eqref{eq:reflectedSDE} with the observation process defined by
\begin{equation}
\label{eq:obSto}
\dd Y^\varepsilon_t = h ( t , X^\varepsilon_t) \dd t + \sqrt{\varepsilon} \dd B'_t, \quad Y^\varepsilon_0 =0,
\end{equation}
where $(B'_t)_{0\leq t \leq T}$ is an adapted Brownian motion in $\RR^{m}$ independent of $(B_t)_{0\leq t \leq T}$.
Equation \eqref{eq:obSto} is the stochastic analogous of \eqref{eq:obs}.
The purpose of stochastic filtering is to compute the law $\pi^\varepsilon_t$ of $X^\varepsilon_t$ knowing the observation up to time $t$, i.e. the law of $X^\varepsilon_t$ conditionally to the filtration $\sigma( Y^\varepsilon_s, 0 \leq s \leq t )$:
\[ \forall \varphi \in C_b(\overline{\stateSpace},\R), \quad \int_{\overline{\stateSpace}} \varphi \, \dd \pi^\varepsilon_t = \E \big [ \varphi ( X^\varepsilon_t ) \vert \sigma( Y^\varepsilon_s, 0 \leq s \leq t ) \big]. \] 
Hence, $\int_{\overline{\stateSpace}}\varphi \, \dd \pi^\varepsilon_t$ is the optimal estimator in the least-square sense of $\varphi( X^\varepsilon_t )$ given $(Y^\varepsilon_s )_{0 \leq s \leq t}$.
From \cite[Section 3]{pardoux1978filtrage}, $\pi^\varepsilon_t$ is expected to have a density w.r.t. the Lebesgue measure, and thus $\pi^\varepsilon_t$ does not charge the boundary $\partial G$. 
We are now interested in the asymptotic behavior of $\pi^\varepsilon_t$ when $\varepsilon$ goes to $0$ . 
In particular, we show that $\pi^\varepsilon_t$ concentrates on minimisers of $x \mapsto \costcome (t,x)$ as $\varepsilon \rightarrow 0$, 
corresponding to the Mortensen observer $\hat{x}(t)$ when it is uniquely defined. 
Note that for each $t \geq 0$, there exists at least one minimiser for the function  $x \mapsto \costcome (t,x)$ since it is continuous on the compact set $\overline{G}$.
We first review the existing results for non-constrained systems in Sections \ref{subsubsec:KalmanSto}-\ref{subsubsec:BarasJames}, before stating our new results for the non-smooth case in Section \ref{subsubsec:SmallReflected}.

\subsubsection{Kalman-Bucy filter} \label{subsubsec:KalmanSto}

We go back to the linear setting of Section \ref{subsec:IntroKalman},
\[ b(t,x) = A x, \qquad \sigma(t,x) = \Sigma, \qquad h(t,x) = H x, \]
and we initialise $X^\varepsilon_0$ from the Gaussian law $\mathcal{N}(\hat{x}_0,\varepsilon P_0)$. 
The processes defined by \eqref{eq:reflectedSDE} and \eqref{eq:obSto} are then Gaussian at each time, so that the same goes for the conditioned process: it is thus sufficient to compute the conditional mean and covariance matrix $(\hat{X}^\varepsilon_t,P^\varepsilon_t)$. 
Classically \cite{kalmanbucy1961,jazwinski1970stochastic,davis1977linear}, the covariance $P^\varepsilon_t = \varepsilon P(t)$ is deterministic, and 
\begin{equation} \label{eq:kalmanSto}
\dd \hat{X}^\varepsilon_t = A \hat{X}^\varepsilon_t \dd t + P(t)H^\top [\dd Y^\varepsilon_t - H \hat{X}^\varepsilon_t \dd t], \quad \hat{X}^\varepsilon_0 = \hat{x}_0,  
\end{equation}
where $P(t)$ is the solution to the Riccati equation \eqref{eq:Riccati}. 
As a consequence, 
\[ \pi^\varepsilon_t ( \dd x ) = Z_\varepsilon^{-1} \exp \bigg[ - \frac{[x - \hat{X}^\varepsilon_t] \cdot P^{-1}(t)[ x - \hat{X}^\varepsilon_t ]}{2 \varepsilon} \bigg] \dd x, \]
where $Z_\varepsilon$ is a normalisation constant. 
As $\varepsilon \rightarrow 0$, this density concentrates on the solution $\hat{x}(t)$ of the deterministic equation \eqref{eq:kalman}.
If we freeze a $C^1$ realisation $(y(t))_{t \geq 0}$ of the process $(Y^\varepsilon_t )_{t \geq 0}$, Equation \eqref{eq:kalmanSto} exactly correspond to \eqref{eq:kalman}, and the related quadratic cost-to-come corresponds to the logarithm of the density of $\pi^\varepsilon_t$. 
The Mortensen observer in this linear setting thus appears as a maximum-likelihood estimator.

\subsubsection{Small noise filtering for the penalised dynamics} \label{subsubsec:BarasJames}
Let us apply the results of \cite{barasjames} to (a variation of) the penalised dynamics \eqref{eq:dynsyspen} from Section \ref{subsec:ResPen}, in the case $\NstateSpace = \NmodelNoiseSpace$ and $\sigma \equiv \mathrm{Id}$.
In this setting, the analogous of \eqref{eq:reflectedSDE}-\eqref{eq:obSto} reads
\[ 
\begin{cases}
\dd X^{\kappa,\varepsilon}_t = b(t,X^{\kappa,\varepsilon}_t) \dd t - f_\kappa (X^{\kappa,\varepsilon}_t) \dd t + \sqrt{\varepsilon} \dd B_t, \\
\dd Y^{\kappa,\varepsilon}_t = h (t,X^{\kappa,\varepsilon}_t) \dd t + \sqrt{\varepsilon} \dd B'_t,
\end{cases} 
\]
where $( X^{\kappa,\varepsilon}_t )_{t \geq 0}$ is a diffusion process in $\R^\NstateSpace$ with initial law $q_0^{\kappa,\varepsilon}$.
The results of \cite{barasjames} assume that the coefficients are smooth bounded functions with bounded derivatives, and we assume the same for $b$, $f_\kappa$ and $h$ in this sub-section, up to replacing the penalisation $f_\kappa$ from Section \ref{subsec:ResPen} by  a suitable regularisation. 
Since the present sub-section is mainly here for illustration and comparison purposes, we did not try to extend these assumptions.

The filtering density $\pi^{\kappa,\varepsilon}_t$ can be computed as the solution of a non-linear stochastic PDE known as the Kushner-Stratonovich equation \cite{kushner1967dynamical}.
Alternatively, $\pi^{\kappa,\varepsilon}_t ( \dd x)$ can be computed by normalising the positive (random) measure $q^{\kappa,\varepsilon}_t (x) \dd x$ that solves the Zakai equation \cite{zakai1969optimal,pardouxt1980stochastic}:
\begin{equation} \label{eq:zakai}
\dd q^{\kappa,\varepsilon}_t (x) = (L^{\kappa,\varepsilon}_t)^\star q^{\kappa,\varepsilon}_t (x) \dd t + \frac{1}{\varepsilon} q^{\kappa,\varepsilon}_t  h(t,x) \cdot \dd Y^{\kappa,\varepsilon}_t,
\end{equation}
where $(L^{\kappa,\varepsilon}_t)^\star$ is the formal $L^2$-adjoint of the infinitesimal generator $L^{\kappa,\varepsilon}_t$ of the Markov process $(X^{\kappa,\varepsilon}_t)_{t \geq 0}$, which is given on $C^2$ test functions $\varphi : \RR^n \rightarrow \RR$ by
\[ L^{\kappa,\varepsilon}_t \varphi (x) = [ b(t,x) - f_\kappa (x) ] \cdot \nabla \varphi(x) + \frac{\varepsilon}{2} \Delta \varphi (x).  \]
Following the theory of pathwise filtering \cite{davis1981pathwise}, we introduce the random function $p^{\kappa,\varepsilon}_t$ defined by
\begin{equation} \label{eq:dossSuss}
p^{\kappa,\varepsilon}_t(x) \triangleq \bigg[ - \frac{1}{\varepsilon} Y^{\kappa,\varepsilon}_t h(t,x) \bigg] q^{\kappa,\varepsilon}_t(x),
\end{equation}
which solves the robust Zakai equation 
\begin{multline}\label{eq:zakaiRobust}
\partial_t p^{\kappa,\varepsilon}_t (x) = - [ b(t,x) - \regpen(x) - \nabla h(t,x)Y^{\kappa,\varepsilon}_t] \cdot \nabla p^{\kappa,\varepsilon}_t + \frac{\varepsilon}{2} \Delta p^{\kappa,\varepsilon}_t - \frac{1}{\varepsilon} \bigg\{ \frac{1}{2} \vert h(t,x) \vert^2 \\ 
+ Y^{\kappa,\varepsilon}_t \cdot [ \partial_t + L^{\kappa,\varepsilon}_t] h(t,x) - \frac{1}{2} \vert \nabla h(t,x) Y^\varepsilon_t \vert^2 + \varepsilon \nabla \cdot [ b(t,x) - \regpen(x) - \nabla h(t,x) Y^\varepsilon_t ] \bigg\}.
\end{multline} 
For each continuous realisation $y = ( y(t) )_{t \geq 0}$ of the process $(Y^{\kappa,\varepsilon}_t )_{t \geq 0}$, equation \eqref{eq:zakaiRobust} has a unique (deterministic) solution $ p^{\kappa,\varepsilon}$ which continuously depends on $y$ \cite[Lemma 3.2]{fleming1982optimal}.
This allows us to recover the random function $p^{\kappa,\varepsilon}_t$ by solving \eqref{eq:zakaiRobust} for each continuous path $y$.
This result motivates the approach of \cite{barasjames}: they freeze a $C^1$ realisation $(y(t))_{t \geq 0}$ of the process $(Y^{\kappa,\varepsilon}_t )_{t \geq 0}$, before to study the related solution $p^{\kappa,\varepsilon}$ of \eqref{eq:zakaiRobust} to make a connection with $\costcomepen$ as $\varepsilon \rightarrow 0$.  
Alternatively, Equation \eqref{eq:dossSuss} gives a meaning to $q^{\kappa,\varepsilon}$ from $p^{\kappa,\varepsilon}$  for each continuous realisation $y = ( y(t) )_{t \geq 0}$ of $(Y^{\kappa,\varepsilon}_t )_{t \geq 0}$: $q^{\kappa,\varepsilon}$ is now a continuous function of $y$.
\cite[Section 5]{fleming1997deterministic} then freezes a $C^1$ realisation $y$ to obtain that
\begin{equation} \label{eq:zakaiStrato}
\partial_t q^{\kappa,\varepsilon} (t,x) = (L^{\kappa,\varepsilon}_t)^\star q^{\kappa,\varepsilon} (t,x) \dd t - \frac{1}{\varepsilon} \bigg[ \frac{1}{2} \vert h(t,x) \vert^2 - \dot{y}(t) \cdot h (t,x) \bigg] q^{\kappa,\varepsilon} (t,x),
\end{equation}
which corresponds to the Stratonovich form of \eqref{eq:zakai} where $(Y^{\kappa,\varepsilon}_t )_{t \geq 0}$ was replaced by $(y(t) )_{t \geq 0}$. Now, let us introduce
\[ \tilde{q}^{\kappa,\varepsilon} (t,x) \triangleq \exp \bigg[ -\frac{1}{2 \varepsilon} \int_0^t \lvert \dot{y}(s) \rvert^2 \dd s \bigg] {q}^{\kappa,\varepsilon} (t,x),\quad \costcome^{\kappa,\varepsilon}(t,x) \triangleq - \varepsilon \log \tilde{q}^{\kappa,\varepsilon}(t,x). \]
The exponential factor is a normalisation term that does not affect the minimisation of $x \mapsto \costcome^{\kappa,\varepsilon}(t,x)$.
The following result \cite[Lemma 5.1]{fleming1997deterministic} makes the connection with the penalised cost-to-come $\costcomepen$ defined in \eqref{eq:costcomeIntro}.

\begin{theorem}[Small noise limit]\label{thm:flemSmallNoise}
For every compact set $K \subset \RR^\NstateSpace$, if
\begin{equation} \label{eq:iniCVK}
\sup_{x \in K} \lvert -\varepsilon \log q^{\kappa,\varepsilon}_0 (x) - \psi(x) \rvert \leq C_K \varepsilon^{1/2}, 
\end{equation} 
for some $C_K >0$ independent of $\varepsilon$, then for every $t > 0$,
\[ \sup_{(s,x) \in [0,t] \times K} \lvert \costcome^{\kappa,\varepsilon} (s,x) - \costcomepen (s,x) \rvert \leq C'_K \varepsilon^{1/2}, \]
for some constant $C'_K > 0$ that only depends on $(t,\kappa,K)$. 
\end{theorem}

As a consequence (see e.g. \cite[Lemma 6.1]{barasjames}), we get the following large deviation result.

\begin{corollary}[Laplace principle] \label{cor:LaplaceJB}
If \eqref{eq:iniCVK} holds for every compact set, then for every bounded continuous $\Phi :\RR^\NstateSpace \rightarrow \RR$, 
\[ \forall t \geq 0, \quad -\varepsilon \log \int_{\RR^\NstateSpace} e^{-\Phi(x)/\varepsilon} \tilde{q}^{\kappa,\varepsilon}(t,x) \dd x \xrightarrow[\varepsilon \rightarrow 0]{} \inf_{x \in \RR^\NstateSpace} \Phi(x) + \costcomepen (t,x). \]
\end{corollary}

This statement is equivalent to a large deviation principle for $\tilde{q}^{\kappa,\varepsilon} (t,x) \dd x$, see e.g. \cite{dembo2009large,feng2006large}.
As $\varepsilon \rightarrow 0$, this tells that the non-normalised density $q^{\varepsilon,\kappa}_t$ concentrates on the minimisers of $x \mapsto \costcomepen(t,x)$.
When uniqueness holds for this minimiser, this precisely corresponds to the Mortensen observer \eqref{eq:mortensen} as defined in Section \ref{subsec:IntroMortensen}.
This extends the observation that we made in the Gaussian setting of Section \ref{subsubsec:KalmanSto}.
A stronger large deviation result for the conditional density $q^{\kappa,\varepsilon}_t$ given $y$ is proved in \cite{HijabAoP}.

\subsubsection{Small noise filtering for reflected dynamics} \label{subsubsec:SmallReflected}
Still in the setting $\NstateSpace = \NmodelNoiseSpace$ and $\sigma \equiv \mathrm{Id}$, we now extend the previous results to the reflected dynamics:
\[ 
\begin{cases}
\dd X^{\varepsilon}_t + \partial\model (X^{\varepsilon}_t) ( \dd t) \ni b (t,X^{\varepsilon}_t) \dd t + \sqrt{\varepsilon} \dd B_t, \\
\dd Y^{\varepsilon}_t = h (t,X^{\kappa,\varepsilon}_t) \dd t + \sqrt{\varepsilon} \dd B'_t,
\end{cases} 
\]
the initial law of $x$ being $q^\varepsilon_0 (x) \dd x$.
As previously, the filtering density $\pi^\varepsilon_t(\dd x)$ can be computed by normalising the solution $q^{\varepsilon}_t (x) \dd x$ of the Zakai equation. 
In this setting, the Zakai equation was studied in \cite{pardoux1978filtrage,pardoux1978stochastic,hucke1990nonlinear}.
Equation \eqref{eq:zakai} is now completed with a ``no-flux'' boundary condition:
\[
\begin{cases}
\dd q^\varepsilon_t(x) = \nabla \cdot [ - q^\varepsilon_t(x) b(t,x) + \frac{\varepsilon}{2} \nabla q^\varepsilon_t(x) ] + \frac{1}{\varepsilon} q^\varepsilon_t(x) h(x) \cdot \dd Y^\varepsilon_t, \quad x \in \stateSpace, \\
- q^\varepsilon_t (x) b(t,x) \cdot n(x)  + \frac{\varepsilon}{2} \frac{\partial q^\varepsilon_t}{\partial n}(x) = 0, \quad x \in \partial \stateSpace,
\end{cases}
\]
As previously, we freeze a $C^1$ realisation $y = (y(t))_{t \geq 0}$ to bring the study back to a deterministic function $q^\varepsilon$ that continuously depends on $y$. 
After normalising 
\[ \tilde{q}^{\varepsilon} (t,x) \triangleq \exp \bigg[ -\frac{1}{2 \varepsilon} \int_0^t \lvert \dot{y}(s) \rvert^2 \dd s \bigg] {q}^{\varepsilon} (t,x), \]
we end up with the following equation:
\begin{equation} \label{eq:RefZakaiStrato}
\begin{cases}
\partial_t \tilde{q}^{\varepsilon} (t,x) = \nabla \cdot [ - \tilde{q}^\varepsilon_t(x) b(t,x) + \frac{\varepsilon}{2} \nabla \tilde{q}^\varepsilon_t(x) ] - \frac{1}{2\varepsilon}  \vert \dot{y}(t) - h(t,x) \vert^2 \tilde{q}^{\varepsilon} (t,x), \\
- \tilde{q}^\varepsilon (t,x) b(t,x) \cdot n(x)  + \frac{\varepsilon}{2} \frac{\partial \tilde{q}}{\partial n}(t,x) = 0, \quad x \in \partial \stateSpace.
\end{cases}
\end{equation}

We assume that $q^\varepsilon_0$ is defined and $C^1$ in a neighbourhood of $\partial \stateSpace$, so that existence and uniqueness for a solution of \eqref{eq:RefZakaiStrato} in $C([0,T] \times \overline\stateSpace) \cap C^{1,2}((0,T] \times \stateSpace)$ stems from \cite[Chapter 5,Corollary 2]{friedman2008partial}.
If we define $\costcome^\varepsilon \triangleq -\varepsilon \log \tilde{q}^\varepsilon$ as previously, we obtain the boundary condition 
\[ b(t,x) \cdot n(x) + \frac{1}{2} \frac{\partial \costcome^\varepsilon}{\partial n} (t,x) = 0, \quad x \in \partial \stateSpace, \]
which differs from the one of the sub-solution in Theorem \ref{thm:Visco}.
This suggests that the analogous of Theorem \ref{thm:flemSmallNoise} is no more true in the present setting (except if $b \cdot n \leq 0$ on $\partial G$).
However, we still managed to prove the large deviation result corresponding to Corollary \ref{cor:LaplaceJB}. 

\begin{theorem}[Laplace principle] \label{thm:refLaplace}
If
\begin{equation*} \label{eq:iniCVG}
\sup_{x \in \overline\stateSpace} \lvert -\varepsilon \log q^{\varepsilon}_0 (x) - \psi(x) \rvert \xrightarrow[\varepsilon \rightarrow 0]{} 0, 
\end{equation*} 
then for every continuous $\Phi :\overline\stateSpace \rightarrow \RR$, 
\[ \forall t \geq 0, \quad -\varepsilon \log \int_{\stateSpace} e^{-\Phi(x)/\varepsilon} \tilde{q}^{\varepsilon} (t,x) \dd x \xrightarrow[\varepsilon \rightarrow 0]{} \inf_{x \in \overline\stateSpace} \Phi(x) + \costcome (t,x). \]
\end{theorem}

As $\varepsilon \rightarrow 0$, this tells that the non-normalised density $q^{\varepsilon}_t$ concentrates on the minimisers of $x \mapsto \costcome (t,x)$, recovering the desired Mortensen observer.
Theorem \ref{thm:refLaplace} is proved in Section \ref{sec:Limitfiltering}.

\begin{remark}[Loss of the boundary condition]
Interestingly, the boundary condition for $\costcome^\varepsilon$ differs from the one of the sub-solution in Theorem \ref{thm:Visco}.
This reminds us of a fundamental difference between $X^\varepsilon$ and $x_\modelNoise$ around the boundary $\partial \stateSpace$.
Indeed, the time spent by $X^\varepsilon$ on $\partial \stateSpace$ has $0$ Lebesgue-measure, whereas $x_\modelNoise$ can have a continuous dynamics on $\partial \stateSpace$.
This is also related to the non-reversibility of non-smooth dynamics.
\end{remark}

\section{Study of the value function} \label{sec:HJB}

This section is devoted to the proofs of Theorems \ref{thm:CVCostpen} and \ref{thm:Visco}, i.e. the uniform convergence of $\costcomepen$ and the viscosity solution description of $\costcome$.

\subsection{Convergence of the penalised problem} \label{subsec:CVpen}

\begin{lemma}[Control of the penalisation] \label{lem:Unifpenbound}
For all $(x_\modelNoise(0),\modelNoise)\in \overline\stateSpace \times L^2((0,t),\R^\NmodelNoiseSpace)$:
\begin{enumerate}
    \item[(i)]\label{eq:penunif} $\sup_{\kappa > 0} \sup_{0 \leq s \leq t} \lvert x^\penal_{\modelNoise}(s) \rvert \leq C$,
    \item[(ii)] $\sup_{\kappa > 0} \sup_{0 \leq s \leq t} \int_0^t \lvert \regpen ( x^\penal_{\modelNoise}(s) ) \rvert \dd s \leq C$, 
    \item[(iii)] $\sup_{0 \leq s \leq t} \mathrm{dist} (x^\penal_{\modelNoise}(s), \overline\stateSpace ) \leq C \kappa^{-1/2}$, 
\end{enumerate}
for a constant $C >0$ that only depends on $(t,\lVert \modelNoise \rVert_{L^2})$.
\end{lemma}

\begin{proof}
First, we write that
\begin{equation} \label{eq:majpen}
\lvert x^\penal_{\modelNoise}(t) \rvert \leq \lvert x_\modelNoise(0) \rvert + \int_0^t \lvert \regpen ( x^\penal_{\modelNoise}(s) ) \rvert + \lvert b(s,x^\penal_{\modelNoise}(s) ) + \sigma (s,x^\penal_{\modelNoise}(s) ) \omega(s) \rvert \dd s. 
\end{equation} 
We define $p(x) \triangleq \tfrac{1}{2} \mathrm{dist}(x,\overline\stateSpace)^2 = \tfrac{1}{2} (x - \proj(x))^2$.
Since $\overline\stateSpace$ is closed and convex, $p$ is $C^1$ with $\nabla p = \penal^{-1} \regpen$, so that 
\begin{equation*} 
p( x^\penal_{\modelNoise}(t) ) + \penal \int_0^t \lvert \nabla p( x^\penal_{\modelNoise}(s) ) \rvert^2 \dd s
= \int_0^t \nabla p ( x^\penal_{\modelNoise}(s) ) \cdot [ b(s,x^\penal_{\modelNoise}(s) ) + \sigma (s,x^\penal_{\modelNoise}(s) ) \omega(s) ] \dd s, \end{equation*}
and then 
\begin{subequations}
\begin{align}
	    &\bigg( \int_0^t \lvert \nabla p( x^\penal_{\modelNoise}(s) ) \rvert^2 \dd s \bigg)^{1/2} \leq \penal^{-1} \bigg( \int_0^t \lvert b(s,x^\penal_{\modelNoise}(s) ) + \sigma (s,x^\penal_{\modelNoise}(s) ) \omega(s) \rvert^2 \dd s \bigg)^{1/2},\label{eq:majpencorr}\\
    &p( x^\penal_{\modelNoise}(t) ) \leq \penal^{-1} \bigg( \int_0^t \lvert b(s,x^\penal_{\modelNoise}(s) ) + \sigma (s,x^\penal_{\modelNoise}(s) ) \omega(s) \rvert^2 \dd s \bigg)^{1/2}.\label{eq:majpendist}
\end{align}
\end{subequations}
Plugging the square of \eqref{eq:majpencorr} into the square of \eqref{eq:majpen}, the Gronwall lemma yields $(i)$ using that coefficients are Lipschitz. 
Plugging $(i)$ into \eqref{eq:majpencorr}-\eqref{eq:majpendist} then yields $(ii)$-$(iii)$. 
\end{proof}

\begin{lemma}[Convergence of penalised curves] \label{lem:pentrajCV}
For every $(x_\modelNoise(0),\modelNoise)$ in $\overline\stateSpace \times L^2((0,t),\R^\NmodelNoiseSpace)$,
\[ \sup_{0 \leq s \leq t} \lvert  x^\penal_{\modelNoise}(s) -  x_{\modelNoise}(s) \rvert \leq C \kappa^{-1/4}, \]
for a constant $C >0$ that only depends on $(t,\lVert \modelNoise \rVert_{L^2})$.
\end{lemma}

\begin{proof} 
We define the curves $y$, $\overline{y}^\penal$ and $y^\penal$ by
\begin{subequations}
\begin{align*}
	   &y(t) \triangleq x_\modelNoise(0) + \int_0^t b(s,x_{\modelNoise}(s) ) + \sigma (s,x_{\modelNoise}(s) ) \omega(s) \dd s, \\
    &\overline{y}^\penal(t) \triangleq x_\modelNoise(0) + \int_0^t b(s,x^\penal_{\modelNoise}(s) ) + \sigma (s,x^\penal_{\modelNoise}(s) ) \omega(s) \dd s, \\
    &y^\penal(t) \triangleq \proj(x^\penal_{\modelNoise}(t)) - x^\penal_{\modelNoise}(t) + \overline{y}^\penal(t),
\end{align*}
\end{subequations}   
together with the correction terms
\[ \varphi(t)\triangleq x_{\modelNoise}(t) - y(t), \qquad \varphi^\penal (t) \triangleq -\int_0^t \regpen (x^\penal_{\modelNoise}(s)) \dd s.\]
We then apply \cite[Lemma 2.2]{tanaka1979stochastic} to $\proj(x^\penal_{\modelNoise})= y^\penal + \varphi^\penal$ and $x_{\modelNoise} = y + \varphi$:
\begin{align*}
\lvert \proj(x^\penal_{\modelNoise}(t)) &- x_{\modelNoise}(t) \rvert^2
\leq \lvert y^\penal (t) - y(t) \rvert^2 + 2\!\!\int_0^t [ y^\penal(t) - y(t) - y^\penal (s) + y(s) ] \dd(\varphi^k -\varphi)(s) \\
&\hspace{1cm}\leq \lvert y^\penal (t) - y(t) \rvert^2 
+ 2 \int_0^t [ y^\penal(t) - \overline{y}^\penal(t) - y^\penal (s) + \overline{y}^\penal (s) ] \dd(\varphi^k -\varphi)(s) \\
&\hspace{3.9cm}+ 2 \int_0^t [ \overline{y}^\penal(t) - y(t) - \overline{y}^\penal (s) + y(s) ] \dd(\varphi^k -\varphi)(s).
\end{align*} 
The first integral is bounded by
\[ 2 \sup_{0 \leq s \leq t} \lvert \proj(x^\penal_{\modelNoise}(s)) - x^\penal_{\modelNoise}(s) \rvert \int_0^t \dd \lvert \varphi^\penal \rvert (s) + \dd \lvert \varphi \rvert (s) \leq C \kappa^{-1/2}, \]
where we used Lemma \ref{lem:Unifpenbound}-$(iii)$, and
Lemma \ref{lem:Unifpenbound}-$(ii)$ to bound $$\int_0^t \dd \lvert \varphi^\penal \rvert (s) = \int_0^t \lvert \regpen ( x^\penal_{\modelNoise}(s) ) \rvert \dd s$$ uniformly in $\penal$.
Integrating by parts, the second integral becomes
\begin{align*}
\int_0^t [ \varphi^k(s) - \varphi(s) ] \dd (\overline{y}^\penal &- y)(s) = \int_0^t [ x^\penal_{\modelNoise}(s) - x_{\modelNoise}(s) - \overline{y}^\penal(s) + y(s) ] \dd (\overline{y}^\penal - y)(s) \\
&= \int_0^t [ x^\penal_{\modelNoise}(s) - x_{\modelNoise}(s) ] \dd (\overline{y}^\penal - y)(s) - \frac{1}{2} \lvert \overline{y}^\penal(t) - y(t) \rvert^2.
\end{align*} 
We now notice that $\overline{y}^\penal -y$ is absolutely continuous w.r.t. the Lebesgue measure, and we bound the remaining integral by
\begin{equation*}
2 \int_0^t \lvert x^\penal_{\modelNoise}(s) - x_{\modelNoise}(s) \rvert^2 + \lvert b(s,x^\penal_{\modelNoise}(s))
+ \sigma (s,x^\penal_{\modelNoise}(s)) \omega (s) - b(s,x_{\modelNoise}(s)) - \sigma (s,x_{\modelNoise}(s)) \omega (s) \rvert^2 \dd s,
\end{equation*} 
where we also used that $ab \leq 2(a^2 + b^2)$. 
Lemma \ref{lem:Unifpenbound}-$(iii)$ gives $C' > 0$ such that
\[ \lvert x^\penal_{\modelNoise}(s) - x_{\modelNoise}(s) \rvert \leq C' \kappa^{-1/2} + \lvert \proj ( x^\penal_{\modelNoise}(s) ) - x_{\modelNoise}(s) \rvert. \]
Gathering all the terms and using that coefficients are Lipschitz, the Gronwall lemma gives the desired result.
\end{proof}

\begin{corollary}[Time-regularity] \label{cor:Hölder}
For every $(x_\modelNoise(0),\modelNoise)$ in $\overline\stateSpace \times L^2((0,t),\R^\NmodelNoiseSpace)$,
\begin{itemize}
    \item[(i)] $\forall s \in [0,t], \quad \int_s^t \rvert \dot{x}_\omega(r) \vert^2 \dd r \leq C (t-s) + \int_s^t \lvert \omega(r) \rvert^2 \dd r,$
    \item[(ii)] $\forall r,s \in [0,t], \quad \vert x_{\modelNoise}(r) - x_{\modelNoise}(s) \vert \leq C \vert r- s \vert^{1/2},$
\end{itemize}
for a constant $C >0$ that only depends on $(t,\lVert \modelNoise \rVert_{L^2})$.
\end{corollary}

\begin{proof}
Starting from the penalised dynamics:
\[ \dot{x}^\penal_{\modelNoise}(r) = -\regpen ( x^\penal_{\modelNoise}(r) ) + b(r,x^\penal_{\modelNoise}(r) ) + \sigma (r,x^\penal_{\modelNoise}(r) ) \omega(r). \]
We now take the square and we integrate.
Using Lemma \ref{lem:Unifpenbound}-$(i)$ and \eqref{eq:majpencorr} to bound the penalisation, we obtain a bound on $x^\penal_{\modelNoise}$ in $H^1((s,t),\RR^\NstateSpace)$ that does not depend on $\kappa$.
Since $x^\penal_{\modelNoise}$ uniformly converges towards $x_{\modelNoise}$, this bound implies weak convergence towards $x_{\modelNoise}$ in $H^1((s,t),\RR^\NstateSpace)$. 
The $H^1$-norm being lower semi-continuous w.r.t. weak convergence, this proves $(i)$.  
Using the Cauchy-Schwarz inequality, $(ii)$ is a consequence of $(i)$. 
\end{proof}

\begin{proof}[Proof of Theorem \ref{thm:CVCostpen}]
From Lemma \ref{lem:pentrajCV}-$(i)$, $\costcomepen$ can be bounded by a constant $M > 0$ uniformly in $\kappa$.
As a consequence, we can restrict the minimisation \eqref{eq:costcomepen} to controls $\modelNoise$ with square $L^2$-norm lower than $2 M$. 
We then plug the result of Lemma \ref{lem:pentrajCV} into the minimisation \eqref{eq:costcomepen}: since $\psi$ and $h$ are Lipschitz-continuous, this completes the proof.
\end{proof}

\begin{proof}[Proof of Corollary \ref{cor:GammaCV}]
Let $(\hat{x}^\kappa (t) )_{\kappa > 0}$ be a family of minimisers for $x \mapsto \costcomepen (t,x)$.
As in the above proof, we restrict the minimisation \eqref{eq:costcomepen} to controls $\modelNoise$ with square $L^2$-norm lower than $2 M$.
The coefficients being Lipschitz, it is standard to show that $\costcomepen (t,\hat{x}^\kappa (t))$ is realised by some $\omega^\kappa \in L^2 (0,t)$ with $\lVert \omega^\kappa \rVert^2_{L^2(0,t)} \leq 2M$. 
Lemma \ref{lem:Unifpenbound}-$(i)$ then shows that $(\hat{x}^\kappa (t) )_{\kappa > 0}$ is bounded, hence pre-compact.
Since uniform convergence implies $\Gamma$-convergence, the result follows.
\end{proof}

\subsection{Viscosity solution} \label{subsec:visco}

The key-ingredient is the following dynamic programming principle.

\begin{lemma}[Bellman principle] \label{lem:Bellman}
For any $x$ in $\overline\stateSpace$ and $0 \leq t-\tau \leq t \leq T$, the dynamic programming holds:
\begin{equation*}\label{eq:Bellman}
\costcome(t,x) = \inf_{(x_\modelNoise (0), \modelNoise) \in \mathcal{A}^\stateSpace_{t, x}} \costcome (t-\tau,\state_\modelNoise (t-\tau)) + \int_{t-\tau}^t \timeError (s, \state_\modelNoise(s), \modelNoise(s)) \diff s.
\end{equation*}
\end{lemma}

\begin{proof} 
The proof relies on classical arguments; however, given that we are working with backward processes and that the trajectory does not necessarily admit a unique backward solution for each control, we prefer to provide the proof here.

Let $(x_\modelNoise(0),\modelNoise)$ belong to $\mathcal{A}^\stateSpace_{t, x}$. For any $(x_{\modelNoise'}(0),\modelNoise') \in \mathcal{A}_{t-\tau,\state_\modelNoise(t-\tau)}$, we define
\[
\tilde\modelNoise(s) \triangleq
\begin{cases}
\modelNoise'(s) &\text{ if } 0 \leq s \leq t - \tau, \\
\modelNoise(s) &\text{ if } t - \tau < s \leq t.
\end{cases} 
\]
By construction, $(x_{\modelNoise'}(0),\tilde\modelNoise') $ belongs to $\mathcal{A}^\stateSpace_{t, x}$, and according to
\cite{EdmondThibault2005}, there exists a unique 
trajectory $x_{\modelNoise'}$ emanating from $x_{\modelNoise'}(0)$
with the control $ \tilde\modelNoise$.
Hence, $x_{\tilde\modelNoise}\equiv x_\modelNoise'$ on $[0,t-\tau]$ and $x_{\tilde\modelNoise}\equiv x_\modelNoise$ on $[t-\tau,T]$, and we have
\[ \costcome(\state,t) \leq \psi(x_{\modelNoise'}(0)) + \int^{t-\tau}_0 \timeError(s,x_{\modelNoise'}(s),\omega'(s)) + \int_{t-\tau}^t \timeError (\state_\modelNoise(s), \modelNoise(s), s ) \diff s. \]
This inequality being true for every 
 $(x_{\modelNoise'}(0),\modelNoise') \in \mathcal{A}_{t-\tau,\state_\modelNoise(t-\tau)}$, we get that
\[ \costcome(t,\state) \leq \costcome(t-\tau,\state_\modelNoise(t-\tau)) + \int_{t-\tau}^t \timeError (s,\state_\modelNoise(s), \modelNoise(s)) \diff s. \]
Moreover, for every $\varepsilon > 0$, some $(x_{\modelNoise}(0),\modelNoise) \in \mathcal{A}^\stateSpace_{t, x}$ exists such that
\begin{align*}
\costcome(t,\state) &+ \varepsilon \\
&\geq \psi(x_{\modelNoise}(0)) + \int_0^{t- \tau} \timeError(s,x_{\modelNoise}(s),\modelNoise (s)) + \int_{t-\tau}^t \timeError (s,\state_\modelNoise(s), \modelNoise(s) ) \diff s \\
&\geq \costcome(t-\tau,\state_\modelNoise(t-\tau)) + \int_{t-\tau}^t \timeError (s,\state_\modelNoise(s), \modelNoise(s)) \diff s,
\end{align*} 
using that $(x_\modelNoise(0),\modelNoise\rvert_{[0,t-\tau]}) \in \mathcal{A}^\stateSpace_{t-\tau,\state_\modelNoise(t-\tau)}$. 
Since such a $(x_\modelNoise(0),\modelNoise)$ exists for every $\varepsilon>0$, this completes the proof.
\end{proof}
 We now turn to the proof of Theorem \ref{thm:Visco}.
\begin{proof}[Proof for the sub-solution part]
Given $t >0 $ and $x \in \overline\stateSpace$, let $\varphi : \RR_+ \times \overline\stateSpace \rightarrow \RR$ be a $C^1$ test function such that $\costcome-\varphi$ has a local maximum at $(t,x)$.
We now use that the control additively enters the dynamics ($\sigma \equiv \mathrm{Id})$.
If $x \in \stateSpace$, for every every $\tilde\modelNoise \in \RR^\NstateSpace$ we can find $(x_{\modelNoise}(0),\modelNoise) \in \mathcal{A}^\stateSpace_{t,x}$ such that $\modelNoise$ is continuous, $\modelNoise(t) = \tilde\modelNoise$ and $x_\modelNoise(s) \in \stateSpace$ for $s < t$.
If $x \in \partial\stateSpace$, this is still possible provided that $[b(t,x) + \tilde\modelNoise] \cdot n(x) \geq 0$.
From the local maximum condition,
\begin{align*}
\varphi(t,x) - \varphi (t-\tau,  \state_{\modelNoise}( t - \tau ) ) &\leq \costcome(t,x) - \costcome (t-\tau,  \state_{\modelNoise}( t - \tau ) ) \\
&\leq  \int_{t-\tau}^t \timeError ( s,\state_{\modelNoise} (s),{\modelNoise}(s) ) \diff s,   
\end{align*}
where the second inequality stems from Lemma \ref{lem:Bellman}.
By construction of $\omega$, $s \mapsto \state_{\modelNoise}(s)$ is differentiable at $t$.
Dividing by $\tau$ and taking the $\tau \rightarrow 0^+$ limit gives that
\begin{equation*} 
\partial_t \varphi(t,x) + [ b(t,x) + \tilde\modelNoise ] \cdot \nabla \varphi(t,x) - \frac{1}{2} \vert \tilde\modelNoise \vert^2 - \frac{1}{2} \vert \dot{y}(t) - h(t,x) \vert^2 \leq 0. 
\end{equation*}
If $x \in \stateSpace$, maximising over $\tilde\omega \in \RR^\NstateSpace$ gives the sub-solution property; moreover, the maximum is realised by $\tilde\omega = \nabla \varphi(t,x)$.
If $x \in \partial\stateSpace$, from the definition of sub-solution \cite[Section 2]{barles1991fully}, we can assume that $[ b(t,x) + \nabla \varphi(t,x) ] \cdot n(x) \geq 0$.
This allows us to take $\tilde\modelNoise = \nabla \varphi(t,x)$ to realise the maximum. In both cases, we obtained the sub-solution property.
\end{proof}

\begin{proof}[Proof for the super-solution part]
Given $(t,x)$ in $\RR_+ \times \overline\stateSpace$, we consider a $C^1$ test function $\varphi : \RR_+ \times \overline{\stateSpace} \rightarrow \RR$ such that $\costcome- \varphi$ has a local minimum at $(t,x)$.
Positive numbers $\delta,h>0$ exist such that 
\begin{equation} \label{MinVisc}
 \vert t-t'\vert \leq \delta \text{ and } \vert x-x'| \leq h \Rightarrow \costcome(t',x') - \varphi(t',x') \geq \costcome ( t,x ) - \varphi ( t,x). 
\end{equation}    
Fix now $\varepsilon > 0$.
Since $\costcome$ is bounded on $\overline{\stateSpace}$ by some $M >0$, we can restrict the minimisation \eqref{eq:costcome} to controls $\modelNoise$ with square $L^2$-norm lower than $2 M$.
From Corollary \ref{cor:Hölder}-$(ii)$,
$s \mapsto x_\modelNoise(s)$ is continuous on $[0,t]$ uniformly in $(x_\modelNoise(0),\modelNoise) \in \mathcal{A}^\stateSpace_{t,x}$ such that $\lVert \modelNoise \rVert^2_{L^2(0,t)} \leq 2M$.
This provides $\eta > 0$ such that for every $(x_\modelNoise(0),\modelNoise) \in \mathcal{A}^\stateSpace_{t,x}$ with $\lVert \modelNoise \rVert^2_{L^2(0,t)} \leq 2M$,
\[  0 \leq \tau \leq \eta \Rightarrow | \state_{\modelNoise}(t-\tau) - \state | \leq h. \]
Let $(\tau_k)_{k \geq 1}$ be a sequence that converges to $0$ with $0 < \tau_k \leq \min ( \delta , \eta) $. 
In Lemma \ref{lem:Bellman}, it is sufficient to minimise over $(x_\modelNoise(0),\modelNoise) \in \mathcal{A}^\stateSpace_{t,x}$ with $\lVert \modelNoise\rVert^2_{L^2(0,t)} \leq 2M$.
By definition of the infimum, there exists $(x_{\modelNoise_k}(0),\modelNoise_k) \in \mathcal{A}^\stateSpace_{t,x}$ such that $\lVert \modelNoise_k \rVert^2_{L^2(0,t)} \leq 2M$ and
\begin{equation} \label{eq:unifSquareOm}
\costcome(t,x) + \varepsilon \tau_k \geq \costcome ( t-\tau_k, \state_{\modelNoise_k} (t-\tau_k) ) + \int_{t-\tau_k}^t \timeError (s, \state_{\modelNoise_k} (s),\modelNoise_k (s) ) \diff s.
\end{equation} 
Using \eqref{MinVisc}, we get that
\begin{align*} 
\varphi (t,x) - \varphi ( t-\tau_k,\state_{\modelNoise_k} (t-\tau_k) ) &\geq \costcome ( t,x) - \costcome (t-\tau_k,\state_{\modelNoise_k} (t-\tau_k) ) \\
&\geq - \varepsilon \tau_k + \int_{t-\tau_k}^t \timeError (s, \state_{\modelNoise_k}(s),\omega_k (s) ) \diff s.
\end{align*}
Since $\varphi$ is $C^1$ and $s \mapsto \state_{\modelNoise_k}(s)$ is absolutely continuous, this yields
\begin{equation} \label{eq:AbsCont}
\int_{t-\tau_k}^t \partial_s \varphi ( s,\state_{\modelNoise_k}(s) ) +\dot{\state}_{\modelNoise_k} (s) \cdot \nabla \varphi (s, \state_{\modelNoise_k}(s))
- \timeError ( s,\state_{\modelNoise_k}(s),\modelNoise_k (s)) \diff s \geq - \varepsilon \tau_k. 
\end{equation}
We now carefully handle the boundary.

\medskip

\emph{\textbf{Case $x \in \stateSpace$:}} the uniform in $k$ continuity of $x_{\omega_k}$ provides that for large enough $k$,
\[ \forall s \in [t - \tau_k, t], \quad \state_{\modelNoise_k}(s) \in \stateSpace, \quad \text{hence} \quad \dot\state_{\modelNoise_k}(s) = b(s,{\state}_{\modelNoise_k}(s))  + \modelNoise_k (s). \]
Moreover,
\begin{multline*} 
b(s,{\state}_{\modelNoise_k} (s) ) \cdot \nabla \varphi (s, \state_{\modelNoise_k}(s)) + \frac{1}{2} \vert \nabla \varphi (s, \state_{\modelNoise_k}(s)) \vert^2 \\
\geq [ b(s,{\state}_{\modelNoise_k}(s))  + \modelNoise_k (s) ] \cdot \nabla \varphi (s, \state_{\modelNoise_k}(s)) - \frac{1}{2} \vert \modelNoise_k (s) \vert^2,
\end{multline*} 
hence from \eqref{eq:AbsCont},
\begin{multline} \label{eq:SuperInt}
\int_{t-\tau_k}^t \big[ \partial_s \varphi ( s,\state_{\modelNoise_k}(s) ) + b(s,{\state}_{\modelNoise_k} (s) ) \cdot \nabla \varphi (s, \state_{\modelNoise_k}(s)) + \frac{1}{2} \vert \nabla \varphi (s, \state_{\modelNoise_k}(s)) \vert^2 \\
- \frac{1}{2} \vert \observ (s) - h (s, \state_{\modelNoise_k}(s) ) \vert^2 \big] \diff s \geq - \varepsilon \tau_k. 
\end{multline} 
Using again that $s \mapsto \state_{\modelNoise_k}(s)$ is continuous at $s = t$ uniformly in $k$, we divide by $\tau_k$ and we take the $k \rightarrow + \infty$ limit to obtain that
\[ \partial_t \varphi ( t,x ) + b(t,x) \cdot \nabla \varphi (t, x) + \frac{1}{2} \vert \nabla \varphi (t, x) \vert^2 - \frac{1}{2} \vert \observ (t) - h (t, \state ) \vert^2 \geq -\varepsilon. \]
Since this holds for every $\varepsilon >0$, this gives the super-solution property.

\medskip

\emph{\textbf{Case $x \in \partial\stateSpace$ and $\nabla \varphi (t,x) \cdot n(x) > 0$:}} since $\stateSpace$ has a $C^2$ boundary, there exists a neighbourhood $U$ of $x$ in $\R^\NstateSpace$ such that 
\[ \stateSpace \cap U = \{ y \in U \, , \, \gamma(y) < 0 \}, \quad \partial \stateSpace \cap U = \{ y \in U \, , \, \gamma(y) = 0 \}, \]
for a $C^2$ function $\gamma : U \rightarrow \R$ with $n(x) = \nabla \gamma (x)$. We now decompose 
\begin{multline*}
\dot{\state}_{\modelNoise_k} (s) \cdot \nabla \varphi (s,{\state}_{\modelNoise_k} (s)) = [\dot{\state}_{\modelNoise_k} (s) \cdot  \nabla \gamma ({\state}_{\modelNoise_k} (s) )][ \nabla \varphi (s, \state_{\modelNoise_k}(s)) \cdot  \nabla \gamma ({\state}_{\modelNoise_k} (s)] \\
+ \pi^\perp_{\nabla \gamma ({\state}_{\modelNoise_k} (s))} ( \dot{\state}_{\modelNoise_k} (s) ) \cdot \pi^\perp_{ \nabla \gamma ({\state}_{\modelNoise_k} (s))} ( \nabla \varphi (s, \state_{\modelNoise_k}(s)) ), \end{multline*} 
where $\pi_{\nabla \gamma}^\perp$ denotes the orthogonal projection on the hyperplane with normal vector $ \nabla \gamma$. 
To alleviate notations, we write $ \nabla \gamma$ instead of $\nabla \gamma ({\state}_{\modelNoise_k} (s))$.
We have $\pi^\perp_{\nabla \gamma} ( \dot{\state}_{\modelNoise_k} (s) ) = \pi^\perp_{\nabla \gamma} ( b(s,{\state}_{\modelNoise_k} (s)) ) + \pi^\perp_{\nabla \gamma} (\modelNoise_k(s))$ hence
\begin{multline} \label{eq:OptControlpi}
\pi^\perp_{\nabla \gamma} ( b(s,{\state}_{\modelNoise_k} (s) ) ) \cdot \pi^\perp_{\nabla \gamma} ( \nabla \varphi (s, \state_{\modelNoise_k}(s)) ) + \frac{1}{2} \vert \pi^\perp_{\nabla \gamma} ( \nabla \varphi (s, \state_{\modelNoise_k}(s)) ) \vert^2 \\ \geq [ \pi^\perp_{\nabla \gamma} ( b(s,{\state}_{\modelNoise_k}(s)) ) +  \pi^\perp_{\nabla \gamma} ( \modelNoise_k (s) ) ] \cdot \pi^\perp_{\nabla \gamma} ( \nabla \varphi (s, \state_{\modelNoise_k}(s)) ) - \frac{1}{2} \vert \pi^\perp_{\nabla \gamma} (  \modelNoise_k (s) ) \vert^2.
\end{multline} 
On the other hand, the uniform in $k$ continuity guarantees that for large enough $k$,
\[ \forall s \in [t - \tau_k, t], \quad \nabla \varphi (s,\state_{\modelNoise_k}(s)) \cdot {\nabla \gamma} (\state_{\modelNoise_k}(s)) \geq 0, \] 
so that, for almost every $s \in [t - \tau_k, t]$,
\[ [\dot{\state}_{\modelNoise_k} (s) \cdot {\nabla \gamma} ][ \nabla \varphi (s, \state_{\modelNoise_k}(s)) \cdot {\nabla \gamma} ] \leq [ b(s,{\state}_{\modelNoise_k}(s)) \cdot {\nabla \gamma} + \modelNoise_k (s) \cdot {\nabla \gamma} ][ \nabla \varphi (s, \state_{\modelNoise_k}(s)) \cdot {\nabla \gamma} ], \]
and we can reason as in \eqref{eq:OptControlpi}. 
Gathering all the terms in \eqref{eq:AbsCont} gives  \eqref{eq:SuperInt}, and we conclude as before.

\medskip

\emph{\textbf{Case $x \in \partial\stateSpace$ and $\nabla \varphi (t,x) \cdot n(x) = 0$:}} the Cauchy-Schwarz inequality yields
\begin{multline} \label{eq:CauchyVsic}
\bigg\vert \int^t_{t- \tau_k} [\dot{\state}_{\modelNoise_k} (s) \cdot {\nabla \gamma} ][ \nabla \varphi (s, \state_{\modelNoise_k}(s)) \cdot {\nabla \gamma} ] \dd s \bigg\rvert \\
\leq \bigg( \int^t_{t- \tau_k} [\dot{\state}_{\modelNoise_k} (s) \cdot {\nabla \gamma} ]^2 \dd s  \bigg)^{1/2} \bigg( \int_{t-\tau_k}^t [ \nabla \varphi (s, \state_{\modelNoise_k}(s)) \cdot {\nabla \gamma} ]^2 \dd s \bigg)^{1/2},
\end{multline} 
where we write $\nabla \gamma$ instead of ${\nabla \gamma} ( \state_{\modelNoise_k}(s) ) $ as in the previous case.
From Corollary \ref{cor:Hölder}-$(i)$, the first integral on  the r.h.s can be bounded in terms of $\int_{t-\tau_k}^t \lvert \modelNoise_k (s) \rvert^2 \dd s$.
From \eqref{eq:unifSquareOm},
using the continuity of $\costcome$ together with the continuity of $x_{\modelNoise_k}$ uniformly in $k$, this integral goes to $0$ as $k \rightarrow + \infty$. 

In the current viscosity setting, we can always assume that $\varphi$ is $C^2$.
Moreover, $\gamma$ is $C^2$ too.
Using $\nabla \varphi (t,x) \cdot \nabla \gamma (x) = 0$ and Corollary \ref{cor:Hölder}-$(ii)$, this yields $\vert \nabla \varphi (s, \state_{\modelNoise_k}(s)) \cdot {\nabla \gamma} ( \state_{\modelNoise_k}(s) )  \vert^2 \leq C \tau_k$, for $C >0$ independent of $(s,k)$.   
From \eqref{eq:CauchyVsic}, this implies
\[ \tau^{-1}_k \int^t_{t- \tau_k} [\dot{\state}_{\modelNoise_k} (s) \cdot {\nabla \gamma} ( \state_{\modelNoise_k}(s) )  ][ \nabla \varphi (s, \state_{\modelNoise_k}(s)) \cdot {\nabla \gamma} ( \state_{\modelNoise_k}(s) )  ( \state_{\modelNoise_k}(s) )  ] \dd s \xrightarrow[k \rightarrow +\infty]{} 0. \]
From \eqref{eq:AbsCont}-\eqref{eq:OptControlpi} and $\nabla \varphi (t,x) \cdot n (x) = 0$, reasoning as in the previous case concludes.

\medskip

\emph{\textbf{Case $x \in \partial\stateSpace$ and $\nabla \varphi (t,x) \cdot n(x) < 0$:}}  we notice that 
\[ \int_{t-\tau_k}^t \dot{\state}_{\modelNoise_k}(s) \nabla \gamma ( \state_{\modelNoise_k}(s) ) \dd s = \gamma (x) - \gamma ( \state_{\modelNoise_k}(t-\tau_k) ) \geq 0, \]
hence
\begin{multline*}
\int_{t-\tau_k}^t [ \dot{\state}_{\modelNoise_k}(s) \cdot \nabla \gamma ( \state_{\modelNoise_k}(s) ) ] [ \nabla \varphi ( \state_{\modelNoise_k}(s) ) \cdot \nabla \gamma ( \state_{\modelNoise_k}(s) ] \dd s \\
\leq \int_{t-\tau_k}^t [ \dot{\state}_{\modelNoise_k}(s) \cdot \nabla \gamma ( \state_{\modelNoise_k}(s) ) ] [ \nabla \varphi ( \state_{\modelNoise_k}(s) ) \cdot \nabla \gamma ( \state_{\modelNoise_k}(s) ) - \nabla \varphi ( t, x ) \cdot \nabla \gamma ( x ) ] \dd s.
\end{multline*} 
The integral on the r.h.s. can be handled as we did for \eqref{eq:CauchyVsic} to get that
\[ \limsup_{k\rightarrow+\infty} \tau_k^{-1} \int_{t-\tau_k}^t [ \dot{\state}_{\modelNoise_k}(s) \cdot \nabla \gamma ( \state_{\modelNoise_k}(s) ) ] [ \nabla \varphi ( \state_{\modelNoise_k}(s) ) \cdot \nabla \gamma ( \state_{\modelNoise_k}(s)) ] \dd s \leq 0. \]
Going back to \eqref{eq:AbsCont}-\eqref{eq:OptControlpi}, we reason as before, we send $k \rightarrow + \infty$ and then $\varepsilon \rightarrow 0$. 
This gives the incomplete property:
\[ \partial_t \varphi ( t,x ) + \pi^\perp_{n (x)} ( b(t,x)) \cdot \pi^\perp_{n (x)} ( \nabla \varphi (t, x) ) + \frac{1}{2} \vert \pi^\perp_{n (x)} ( \nabla \varphi (t, x) ) \vert^2 - \frac{1}{2} \vert \observ (t) - h (t, \state ) \vert^2 \geq 0. \]
The boundary condition being imposed in the viscosity sense \cite[Section 2]{barles1991fully}, we can assume that $b(t,x) \cdot n (x) + \tfrac{1}{2} \nabla \varphi (t,x) \cdot n(x) \leq 0$, so that 
\[ [ b(t,x) \cdot n (x) ] [ \nabla \varphi(t,x) \cdot n(x) ] + \frac{1}{2} \vert \nabla \varphi (t,x) \cdot n(x) \vert^2 \geq 0, \]
because $\nabla \varphi (t,x) \cdot n(x) \leq 0$. 
Adding this to the incomplete property concludes.
\end{proof}

\begin{remark}[Super-solution property at the boundary] \label{rem:allsuper}
When $b \cdot n \leq 0$ on $\partial G$, we deduce from the last case of the above proof that $\costcome$ is actually a viscosity super-solution of 
\[ \partial_t \costcome (t,x) + b(t,x) \cdot \nabla \costcome (t,x) + \frac{1}{2} \vert \nabla \costcome (t,x) \vert^2 - \frac{1}{2} \vert \dot{y}(t) - h(t,x) \vert^2 = 0, \]
on the full $[0,T] \times \overline{\stateSpace}$ for every $T >0$ (no more boundary condition).
The related comparison principle then enters the scope of \cite[Theorem III.2]{capuzzo1990hamilton}.
\end{remark}

\section{Small noise filtering for reflected dynamics} \label{sec:Limitfiltering}

This section is devoted to the proof of Theorem \ref{thm:refLaplace}.
Our approach is a simple instance of the general machinery developed in \cite{feng2006large}.
Let us a fix a continuous $\Phi : \overline{G} \rightarrow 0$ and $t >0$.
By regularisation and density, we can assume that $\Phi$ is defined and $C^1$ on a neighbourhood of $\overline\stateSpace$.
This allows us to apply \cite[Chapter 5, Corollary 2]{friedman2008partial} to get that 
\begin{equation*}
\begin{cases}
\partial_s \Phi^\varepsilon(s,x) + b(s,x) \cdot \nabla \Phi^\varepsilon(s,x) + \frac{\varepsilon}{2} \Delta \Phi^\varepsilon(s,x)  -\frac{1}{2 \varepsilon} \vert \dot{y}(s) - h(s,x) \vert^2 \Phi^\varepsilon(s,x) = 0,  \\
\Phi^\varepsilon(t,x) = e^{-\Phi(x)/\varepsilon}, \quad x \in \overline\stateSpace, \\
\frac{\partial\Phi^\varepsilon}{\partial n} (s,x) = 0, \quad (s,x) \in [0,t) \times \partial\stateSpace. 
\end{cases}
\end{equation*}
has a unique solution $\Phi^\varepsilon$ in $C([0,t] \times \overline\stateSpace) \cap C^{1,2}([0,t) \times \stateSpace)$.
From \cite[Lemma 3.2]{hucke1990nonlinear}, we get the duality relation
\begin{equation} \label{eq:duality}
\int_\stateSpace \Phi^\varepsilon (0,x) \tilde{q}^\varepsilon(t,x) \dd x = \int_\stateSpace \Phi^\varepsilon (t,x) \tilde{q}^\varepsilon(0,x) \dd x. 
\end{equation} 

The log-transform $V^\varepsilon_\Phi \triangleq -\varepsilon \log \Phi^\varepsilon$ then satisfies
\begin{equation*}
\begin{cases}
\partial_s V^\varepsilon_\Phi (s,x) + b(s,x) \cdot \nabla V^\varepsilon_\Phi (s,x) + \frac{\varepsilon}{2} \Delta V^\varepsilon_\Phi (s,x) - \frac{1}{2} \vert \nabla V^\varepsilon_\Phi (s,x) \vert^2 \\
\qquad\qquad\qquad\qquad+ \frac{1}{2} \vert \dot{y}(s) - h(s,x) \vert^2 = 0, \quad (s,x) \in [0,t) \times \stateSpace, \\ V^\varepsilon_\Phi(t,x) = \Phi(x), \quad x \in \overline\stateSpace, \\
\frac{\partial V_\Phi^\varepsilon}{\partial n} (s,x) = 0, \quad (s,x) \in [0,t) \times \partial\stateSpace. 
\end{cases}
\end{equation*}
From this, we proceed as in \cite[Section 5]{fleming1997deterministic} to give a control representation for $V^\varepsilon_\Phi$.
On a filtered probability space $(\Omega,\mathcal{F},\mathbb{P},(\mathcal{F}_s)_{0 \leq s \leq t})$, let $(B_s)_{0 \leq s \leq t}$ be an adapted Brownian motion. 
For $0 \leq s \leq t$, $x \in \overline\stateSpace$ and any square-integrable adapted process $(\alpha_r)_{s \leq r \leq t}$, we define the controlled reflected dynamics
\[ d Y^{\varepsilon,\alpha}_r + \model ( Y^{\varepsilon,\alpha}_r ) (\dd s) \ni b (s,Y^{\varepsilon,\alpha}_r) \dd r + \alpha_r \dd r + \sqrt{\varepsilon} \dd B_r, \quad s \leq r \leq t, \]
with initial condition $Y^{\varepsilon,\alpha}_s = x$. 
Since $V^\varepsilon_\Phi$ is $C^{1,2}$, a standard verification argument now gives that
\begin{equation} \label{eq:backwardVerif}
V^\varepsilon_\Phi (s,x) = \inf_{\substack{x,\alpha \\ Y^{\varepsilon,\alpha}_s = x}} \E \bigg[ \int_s^t \frac{1}{2} \vert \alpha_r \vert^2 + \frac{1}{2} \vert \dot{y}(r) - h(r,Y^{\varepsilon,\alpha}_r) \vert^2 \dd r + \Phi(Y^{\varepsilon,\alpha}_t) \bigg].
\end{equation}
The next result is the analogous of \cite[Lemma 5.1]{fleming1997deterministic} or Theorem \ref{thm:flemSmallNoise}.

\begin{lemma} \label{lem:RefQuantCV} 
There exists $C > 0$ independent of $\varepsilon$ such that
\[ \sup_{(s,x) \in [0,t] \times \overline\stateSpace} \vert V^\varepsilon_\Phi (s,x) - V^0_\Phi (s,x) \vert \leq C \varepsilon^{1/4}. \] 
\end{lemma}

\begin{proof} 
Since $G$ is bounded, choosing e.g. $\alpha = 0$ in \eqref{eq:backwardVerif}, we get that $V^\varepsilon_\Phi$ is uniformly bounded by some $M >0$ independent of $\varepsilon$.
Thus, we can restrict the minimisation to control processes $( \alpha_r )_{s \leq r \leq t}$ with square $L^2$-norm lower than $2 M$. 
From \cite[Proposition 4.16-I]{pardoux2014sdes} we get that
\[ \E \big[ \sup_{s \leq r \leq t} \vert Y^{\varepsilon,\alpha}_r \vert^2 \big] \leq C_M, \] 
and from \cite[Proposition 4.16-II]{pardoux2014sdes} that
\[ \E \big[ \sup_{s \leq r \leq t} \vert Y^{\varepsilon,\alpha}_r - Y^{0,\alpha}_r \vert^2 \big] \leq C_M \varepsilon^{1/2}, \]
for $C_M > 0$ independent of $\varepsilon$. 
We then plug this into the minimisation \eqref{eq:backwardVerif}: since $\Phi$ and $h$ are Lipschitz-continuous, we conclude by using the Cauchy-Schwarz inequality for the $\frac{1}{2} \vert \dot{y}(r) - h(r,Y^{\varepsilon,\alpha}_r) \vert^2$ term.
\end{proof}

\begin{proof}[Proof of Theorem \ref{thm:refLaplace}]
Going back to \eqref{eq:duality}, the assumption on $q^\varepsilon_0 (x) = \tilde{q}^\varepsilon(0,x)$ and Lemma \ref{lem:RefQuantCV} give that
\[ -\varepsilon \log \int_\stateSpace e^{-\Phi(x)/\varepsilon} q^\varepsilon(t,x) \dd x = -\varepsilon \log \int_\stateSpace e^{-[V^0_\Phi(0,x) + \psi(x) + r^\varepsilon(x)]/\varepsilon} \dd x, \]
for some continuous $r^\varepsilon$ that uniformly converges to $0$ as $\varepsilon \rightarrow 0$. 
The standard Laplace principle (see e.g. \cite[Lemma 6.1]{barasjames}) now shows that the r.h.s. converges towards $\inf_{x \in \overline{\stateSpace}} V^0_\Phi(0,x) + \psi(x)$.
However,
\begin{align*}
\inf_{x \in \overline{\stateSpace}} \, &V^0_\Phi(0,x) + \psi(x) \\
&= \inf_{x_\modelNoise(0), \modelNoise} \Phi(x_\modelNoise(t)) + \psi(x_\modelNoise(0)) + \int_0^t \frac{1}{2} \vert \omega(s) \vert^2 + \frac{1}{2} \vert \dot{y}(s) - h(s,x_\modelNoise(s)) \vert^2 \dd s \\
&= \inf_{x \in \overline{\stateSpace}} \Phi(x) + \costcome(t,x),
\end{align*}
concluding the proof.
\end{proof}

\section*{Acknowledgements}
\addcontentsline{toc}{section}{Acknowledgement}

Philippe Moireau is thankful for support through ANR ODISSE (ANR-19-CE48-0004-01). Laurent Mertz is thankful for support through NSFC Grant No. 12271364 and GRF Grant No. 11302823. 

\printbibliography
\addcontentsline{toc}{section}{References}

\end{document}